\documentclass[10pt,twoside]{article}
\usepackage{my-ndst}

\begin{document}

\setcounter{page}{1}

\newarticle

\title{Some linear and nonlinear integral inequalities\\
on time scales in two independent variables}

\author{R.A.C.~Ferreira\,$^1$ and D.F.M.~Torres\,$^2$}

\address{Department of Mathematics,
University of Aveiro,
3810-193 Aveiro, Portugal}

\renewcommand{\thefootnote}{} \footnote{$^1$ Supported by the Portuguese Foundation for Science and Technology (FCT) through the PhD fellowship SFRH/BD/39816/2007. E-mail: ruiacferreira@ua.pt}

\renewcommand{\thefootnote}{} \footnote{$^2$ Supported by FCT through the R\&D unit CEOC, cofinanced by the EC fund FEDER/POCI
2010. E-mail: delfim@ua.pt (corresponding author)}

\date{}

\abstract{We establish some linear and nonlinear
integral inequalities of Gronwall-Bellman-Bihari type for functions with two independent variables on general time scales.
The results are illustrated with examples, obtained by fixing
the time scales to concrete ones.
An estimation result for the solution
of a partial delta dynamic equation is given
as an application.}



\bigskip


\section{\normalsize Introduction}

Inequalities have always been of great
importance for the development of several
branches of mathematics.
For instance, in approximation theory and numerical analysis,
linear and nonlinear inequalities, in one and more than one
variable, play an important role in the estimation
of approximation errors \cite{Pch}.

Time scales, which are defined as nonempty closed subsets of the
real numbers, are the basic but fundamental ingredient that
permits to define a rich calculus that encompasses both differential and difference tools \cite{Hilger90,Hilger97}.
At the same time one gains more (\textrm{cf.}, \textrm{e.g.},
Corollary~\ref{cor:3.1}).
For an introduction to the calculus on time scales
we refer the reader to \cite{livro}
and \cite{pdiff,miots}, respectively
for functions of one and more than one independent variables.

Integral inequalities of Gronwall-Bellman-Bihari type for functions of a single variable on a time
scale can be found in \cite{inesurvey,Pach,RT,Bihary,Gronwall}.
To the best of the authors knowledge, no such results exist on the literature of time scales when functions of two independent variables are considered. It is our aim to obtain here a first insight on this type of inequalities.


\section{\normalsize Linear inequalities}
\label{sec:mainResults}

Throughout the text we assume that $\mathbb{T}_1$
and $\mathbb{T}_2$ are
time scales with at least two points and consider the time scales
intervals $\tilde{\mathbb{T}}_1=[a_1,\infty)\cap\mathbb{T}_1$ and
$\tilde{\mathbb{T}}_2=[a_2,\infty)\cap\mathbb{T}_2$, for
$a_1\in\mathbb{T}_1$, and $a_2\in\mathbb{T}_2$. We also use the
notations $\mathbb{R}^+_0=[0,\infty)$ and
$\mathbb{N}_0=\mathbb{N}\cup\{0\}$, while
$e_p(t,s)$ denotes the usual exponential function on time scales with $p \in \mathcal{R}$, \textrm{i.e.}, $p$
a regressive function \cite{livro}.

\begin{theorem}
\label{teo1} Let $u(t_1,t_2), a(t_1,t_2), f(t_1,t_2)\in
C(\tilde{\mathbb{T}}_1\times\tilde{\mathbb{T}}_2,\mathbb{R}_0^+)$
with $a(t_1,t_2)$ nondecreasing in each of its variables. If
\begin{equation}
\label{in0} u(t_1,t_2)\leq
a(t_1,t_2)+\int_{a_1}^{t_1}\int_{a_2}^{t_2}
f(s_1,s_2)u(s_1,s_2)\Delta_1 s_1\Delta_2 s_2
\end{equation}
for $(t_1,t_2)\in \tilde{\mathbb{T}}_1\times\tilde{\mathbb{T}}_2$,
then
\begin{equation}
\label{in2} u(t_1,t_2)\leq
a(t_1,t_2) e_{\int_{a_2}^{t_2}f(t_1,s_2)\Delta_2 s_2}(t_1,a_1)\, , \quad (t_1,t_2)\in \tilde{\mathbb{T}}_1\times\tilde{\mathbb{T}}_2.
\end{equation}
\end{theorem}

\begin{proof}
Since $a(t_1,t_2)$ is nondecreasing on
$(t_1,t_2)\in\tilde{\mathbb{T}}_1\times\tilde{\mathbb{T}}_2$,
inequality (\ref{in0}) implies, for an arbitrary
$\varepsilon>0$, that
$$r(t_1,t_2)\leq 1+\int_{a_1}^{t_1}\int_{a_2}^{t_2}
f(s_1,s_2)r(s_1,s_2)\Delta_1 s_1\Delta_2 s_2,$$ where
$r(t_1,t_2)=\frac{u(t_1,t_2)}{a(t_1,t_2)+\varepsilon}$. Define
$v(t_1,t_2)$ by the right hand side of the last inequality. Then,
\begin{equation}
\label{in1} \frac{\partial}{\Delta_2 t_2}\left(\frac{\partial
v(t_1,t_2)}{\Delta_1 t_1}\right)=f(t_1,t_2)r(t_1,t_2)\leq
f(t_1,t_2)v(t_1,t_2),\
(t_1,t_2)\in\tilde{\mathbb{T}}_1^k\times\tilde{\mathbb{T}}_2^k.
\end{equation}
From (\ref{in1}), and taking into account that $v(t_1,t_2)$ is
positive and nondecreasing, we obtain
$$\frac{v(t_1,t_2)\frac{\partial}{\Delta_2 t_2}\left(\frac{\partial
v(t_1,t_2)}{\Delta_1
t_1}\right)}{v(t_1,t_2)v(t_1,\sigma_2(t_2))}\leq f(t_1,t_2),$$
from which it follows that
$$\frac{v(t_1,t_2)\frac{\partial}{\Delta_2 t_2}\left(\frac{\partial
v(t_1,t_2)}{\Delta_1
t_1}\right)}{v(t_1,t_2)v(t_1,\sigma_2(t_2))}\leq
f(t_1,t_2)+\frac{\frac{\partial v(t_1,t_2)}{\Delta_1
t_1}\frac{\partial v(t_1,t_2)}{\Delta_2
t_2}}{v(t_1,t_2)v(t_1,\sigma_2(t_2))}.$$ The previous inequality
can be rewritten as
$$\frac{\partial}{\Delta_2 t_2}\left(\frac{\frac{\partial v(t_1,t_2)}{\Delta_1 t_1}}{v(t_1,t_2)}\right)\leq f(t_1,t_2).$$
Delta integrating with respect to the second variable from $a_2$
to $t_2$ (we observe that $t_2$ can be the maximal element of
$\tilde{\mathbb{T}}_2$, if it exists), and noting that
$\frac{\partial v(t_1,t_2)}{\Delta_1 t_1}\mid_{(t_1,a_2)}=0$, we
have
$$\frac{\frac{\partial v(t_1,t_2)}{\Delta_1 t_1}}{v(t_1,t_2)}\leq\int_{a_2}^{t_2}f(t_1,s_2)\Delta_2 s_2,$$
that is,
$$\frac{\partial v(t_1,t_2)}{\Delta_1 t_1}\leq\int_{a_2}^{t_2}f(t_1,s_2)\Delta_2 s_2 v(t_1,t_2).$$
Fixing $t_2\in\tilde{\mathbb{T}}_2$ arbitrarily, we have that
$p(t_1):=\int_{a_2}^{t_2}f(t_1,s_2)\Delta_2 s_2\in \mathcal{R}^+$. Because $v(a_1,t_2)=1$,
by \cite[Theorem~5.4]{inesurvey}
$v(t_1,t_2)\leq e_p(t_1,a_1)$.
Inequality (\ref{in2}) follows from
$$u(t_1,t_2)\leq [a(t_1,t_2)+\varepsilon]v(t_1,t_2)$$
and the arbitrariness of $\varepsilon$.
$\>\Box$ \end{proof}

\begin{corollary}[\textrm{cf.} Lemma~2.1 of \cite{khellaf}]
\label{cor1} Let $\mathbb{T}_1=\mathbb{T}_2=\mathbb{R}$ and assume
that the functions $u(x,y), a(x,y), f(x,y)\in
C([x_0,\infty)\times[y_0,\infty),\mathbb{R}_0^+)$ with $a(x,y)$
nondecreasing in its variables. If
\begin{equation*}
u(x,y)\leq a(x,y)+\int_{x_0}^{x}\int_{y_0}^{y} f(t,s)u(t,s)dt ds
\end{equation*}
for $(x,y)\in [x_0,\infty)\times[y_0,\infty)$, then
\begin{equation*}
u(x,y)\leq a(x,y)\exp\left(\int_{x_0}^{x}\int_{y_0}^y f(t,s)dt
ds\right)
\end{equation*}
for $(x,y)\in [x_0,\infty)\times[y_0,\infty)$.
\end{corollary}

\begin{corollary}[\textrm{cf.} Theorem~2.1 of \cite{saldisc}]
Let $\mathbb{T}_1=\mathbb{T}_2=\mathbb{Z}$ and assume that the
functions $u(m,n), a(m,n), f(m,n)$ are nonnegative and that
$a(m,n)$ is nondecreasing for $m\in[m_0,\infty)\cap\mathbb{Z}$ and
$n\in[n_0,\infty)\cap\mathbb{Z}$, $m_0,n_0\in\mathbb{Z}$. If
\begin{equation*}
u(m,n)\leq a(m,n)+\sum_{s=m_0}^{m-1}\sum_{t=n_0}^{n-1}
f(s,t)u(s,t)
\end{equation*}
for all $(m,n)\in
[m_0,\infty)\cap\mathbb{Z}\times[n_0,\infty)\cap\mathbb{Z}$, then
\begin{equation*}
u(m,n)\leq
a(m,n)\prod_{s=m_0}^{m-1}\left[1+\sum_{t=n_0}^{n-1}f(s,t)\right]
\end{equation*}
for all $(m,n)\in
[m_0,\infty)\cap\mathbb{Z}\times[n_0,\infty)\cap\mathbb{Z}$.
\end{corollary}

\begin{remark}
We note that, following the same steps of the proof of Theorem~\ref{teo1}, one can obtained other bound on the function $u$, namely
\begin{equation}
\label{in6} u(t_1,t_2)\leq
a(t_1,t_2)e_{\int_{a_1}^{t_1}f(s_1,t_2)\Delta_1 s_1}(t_2,a_2).
\end{equation}
When $\mathbb{T}_1=\mathbb{T}_2=\mathbb{R}$, then the bounds in
(\ref{in2}) and (\ref{in6}) coincide (see Corollary~\ref{cor1}).
If, for example, we let
$\mathbb{T}_1=\mathbb{T}_2=\mathbb{Z}$, the bounds obtained can be different. Moreover, at different points one bound can be sharper than the other and vice-versa (see Example~\ref{ex:rm1}).
\end{remark}

\begin{example}
\label{ex:rm1}
Let $f(t_1,t_2)$ be a function defined by $f(0,0)=1/4$,
$f(1,0)=1/5$, $f(2,0)=1$, $f(0,1)=1/2$, $f(1,1)=0$,
and $f(2,1)=5$. Set $a_1=a_2=0$.
Then, from (\ref{in2}) we get
\begin{equation*}
u(2,1)\leq a(2,1)\frac{3}{2},\quad
u(3,2)\leq a(3,2)\frac{147}{10},
\end{equation*}
while from (\ref{in6}) we get
\begin{equation*}
u(2,1)\leq a(2,1)\frac{29}{20},\quad
u(3,2)\leq a(3,2)\frac{637}{40}.
\end{equation*}
\end{example}

Other interesting corollaries can be obtained from Theorem~\ref{teo1}.

\begin{corollary}
Let $\mathbb{T}_1=q^{\mathbb{N}_0}=\{q^k:k\in\mathbb{N}_0\}$, for
some $q>1$, and $\mathbb{T}_2=\mathbb{R}$. Assume that the
functions $u(t,x),\ a(t,x)$ and $f(t,x)$ satisfy the hypothesis of Theorem~\ref{teo1} for all
$(t,x)\in\tilde{\mathbb{T}}_1\times\tilde{\mathbb{T}}_2$ with
$a_1=1$ and $a_2=0$. If
\begin{equation*}
u(t,x)\leq a(t,x)+\sum_{s=1}^{t/q}(q-1)s\int_{0}^{x}
f(s,\tau)u(s,\tau)d\tau
\end{equation*}
for all $(t,x)\in\tilde{\mathbb{T}}_1\times\tilde{\mathbb{T}}_2$,
then
\begin{equation*}
u(t,x)\leq
a(t,x)\prod_{s=1}^{t/q}\left[1+(q-1)s\int_{0}^{x}f(s,\tau)d\tau\right]
\end{equation*}
for all $(t,x)\in\tilde{\mathbb{T}}_1\times\tilde{\mathbb{T}}_2$.
\end{corollary}

We now generalize Theorem~\ref{teo1}. If in Theorem~\ref{thm:abv} we let $f\equiv 1$ and $g$ not depending on the first two variables, then we obtain
Theorem~\ref{teo1}.

\begin{theorem}
\label{thm:abv}
Let $u(t_1,t_2), a(t_1,t_2), f(t_1,t_2)\in
C(\tilde{\mathbb{T}}_1\times\tilde{\mathbb{T}}_2,\mathbb{R}_0^+)$,
with $a$ and $f$ nondecreasing in each of the variables and
$g(t_1,t_2,s_1,s_2)\in C(S,\mathbb{R}_0^+)$, where
$S=\{(t_1,t_2,s_1,s_2)\in\tilde{\mathbb{T}}_1\times
\tilde{\mathbb{T}}_2\times\tilde{\mathbb{T}}_1\times\tilde{\mathbb{T}}_2:a_1\leq
s_1\leq t_1,a_2\leq s_2\leq t_2\}$. If
\begin{equation*}
u(t_1,t_2)\leq
a(t_1,t_2)+f(t_1,t_2)\int_{a_1}^{t_1}\int_{a_2}^{t_2}
g(t_1,t_2,s_1,s_2)u(s_1,s_2)\Delta_1 s_1\Delta_2 s_2
\end{equation*}
for $(t_1,t_2)\in \tilde{\mathbb{T}}_1\times\tilde{\mathbb{T}}_2$,
then
\begin{equation}
\label{in5} u(t_1,t_2)\leq
a(t_1,t_2)e_{\int_{a_2}^{t_2}f(t_1,t_2)g(t_1,t_2,t_1,s_2)\Delta_2
s_2}(t_1,a_1)\, , \quad (t_1,t_2)\in
\tilde{\mathbb{T}}_1\times\tilde{\mathbb{T}}_2.
\end{equation}
\end{theorem}

\begin{proof}
We start by fixing arbitrary numbers
$t_1^\ast\in\tilde{\mathbb{T}}_1$ and
$t_2^\ast\in\tilde{\mathbb{T}}_2$, and considering the following function defined on
$[a_1,t_1^\ast]\cap\tilde{\mathbb{T}}_1\times[a_2,t_2^\ast]\cap\tilde{\mathbb{T}}_2$
for an arbitrary $\varepsilon>0$:
\begin{equation*}
v(t_1,t_2)=a(t_1^\ast,t_2^\ast)+\varepsilon+f(t_1^\ast,t_2^\ast)\int_{a_1}^{t_1}\int_{a_2}^{t_2}
g(t_1^\ast,t_2^\ast,s_1,s_2)u(s_1,s_2)\Delta_1 s_1\Delta_2 s_2 \, .
\end{equation*}
From our hypothesis we see that
\begin{equation*}
u(t_1,t_2)\leq v(t_1,t_2),\ \mbox{for all}\ (t_1,t_2)\in
[a_1,t_1^\ast]\cap\tilde{\mathbb{T}}_1\times[a_2,t_2^\ast]\cap\tilde{\mathbb{T}}_2.
\end{equation*}
Moreover, delta differentiating with respect to the first variable
and then with respect to the second, we obtain
\begin{align*}
\frac{\partial}{\Delta_2 t_2}\left(\frac{\partial
v(t_1,t_2)}{\Delta_1
t_1}\right)&=f(t_1^\ast,t_2^\ast)g(t_1^\ast,t_2^\ast,t_1,t_2)u(t_1,t_2)\\
&\leq
f(t_1^\ast,t_2^\ast)g(t_1^\ast,t_2^\ast,t_1,t_2)v(t_1,t_2),
\end{align*}
for all $(t_1,t_2)\in
[a_1,t_1^\ast]^k\cap\tilde{\mathbb{T}}_1\times[a_2,t_2^\ast]^k\cap\tilde{\mathbb{T}}_2$.
From this last inequality, we can write
$$\frac{v(t_1,t_2)\frac{\partial}{\Delta_2 t_2}\left(\frac{\partial
v(t_1,t_2)}{\Delta_1
t_1}\right)}{v(t_1,t_2)v(t_1,\sigma_2(t_2))}\leq
f(t_1^\ast,t_2^\ast)g(t_1^\ast,t_2^\ast,t_1,t_2) \, .$$
Hence,
$$\frac{v(t_1,t_2)\frac{\partial}{\Delta_2 t_2}\left(\frac{\partial
v(t_1,t_2)}{\Delta_1
t_1}\right)}{v(t_1,t_2)v(t_1,\sigma_2(t_2))}\leq
f(t_1^\ast,t_2^\ast)g(t_1^\ast,t_2^\ast,t_1,t_2)+\frac{\frac{\partial
v(t_1,t_2)}{\Delta_1 t_1}\frac{\partial v(t_1,t_2)}{\Delta_2
t_2}}{v(t_1,t_2)v(t_1,\sigma_2(t_2))}.$$ The previous inequality
can be rewritten as
$$\frac{\partial}{\Delta_2 t_2}\left(\frac{\frac{\partial v(t_1,t_2)}{\Delta_1 t_1}}{v(t_1,t_2)}\right)\leq f(t_1^\ast,t_2^\ast)g(t_1^\ast,t_2^\ast,t_1,t_2).$$
Delta integrating with respect to the second variable from $a_2$
to $t_2$ and noting that $\frac{\partial v(t_1,t_2)}{\Delta_1
t_1}\mid_{(t_1,a_2)}=0$, we have
$$\frac{\frac{\partial v(t_1,t_2)}{\Delta_1 t_1}}{v(t_1,t_2)}\leq\int_{a_2}^{t_2}f(t_1^\ast,t_2^\ast)g(t_1^\ast,t_2^\ast,t_1,s_2)\Delta_2 s_2,$$
that is,
$$\frac{\partial v(t_1,t_2)}{\Delta_1 t_1}\leq\int_{a_2}^{t_2}f(t_1^\ast,t_2^\ast)g(t_1^\ast,t_2^\ast,t_1,s_2)\Delta_2 s_2 v(t_1,t_2).$$
Fix $t_2=t_2^\ast$ and put
$p(t_1):=\int_{a_2}^{t_2^\ast}f(t_1^\ast,t_2^\ast)g(t_1^\ast,t_2^\ast,t_1,s_2)\Delta_2
s_2\in \mathcal{R}^+$. By \cite[Theorem~5.4]{inesurvey}
$$v(t_1,t_2^\ast)\leq(a(t_1^\ast,t_2^\ast)+\varepsilon) e_p(t_1,a_1).$$
Letting $t_1=t_1^\ast$ in the above inequality, and remembering that $t_1^\ast$, $t_2^\ast$ and $\varepsilon$ are arbitrary, it follows (\ref{in5}).
$\>\Box$ \end{proof}


\section{\normalsize Nonlinear inequalities}

\begin{theorem}
\label{teo2} Let $u(t_1,t_2)$ and $f(t_1,t_2)$ $\in
C(\tilde{\mathbb{T}}_1\times\tilde{\mathbb{T}}_2,\mathbb{R}_0^+)$.
Moreover, let $a(t_1,t_2)\in
C(\tilde{\mathbb{T}}_1\times\tilde{\mathbb{T}}_2,\mathbb{R}^+)$
be a nondecreasing function in each of the variables. If $p$
and $q$ are two positive real numbers such that $p\geq q$ and if
\begin{equation}
\label{in3} u^p(t_1,t_2)\leq
a(t_1,t_2)+\int_{a_1}^{t_1}\int_{a_2}^{t_2}
f(s_1,s_2)u^q(s_1,s_2)\Delta_1 s_1\Delta_2 s_2
\end{equation}
for $(t_1,t_2)\in \tilde{\mathbb{T}}_1\times\tilde{\mathbb{T}}_2$,
then
\begin{equation}
\label{in4} u(t_1,t_2)\leq
a^{\frac{1}{p}}(t_1,t_2)\left[e_{\int_{a_2}^{t_2}f(t_1,s_2)a^{\frac{q}{p}-1}(t_1,s_2)\Delta_2
s_2}(t_1,a_1)\right]^{\frac{1}{p}},\ (t_1,t_2)\in
\tilde{\mathbb{T}}_1\times\tilde{\mathbb{T}}_2.
\end{equation}
\end{theorem}

\begin{proof}
Since $a(t_1,t_2)$ is positive and nondecreasing on
$(t_1,t_2)\in\tilde{\mathbb{T}}_1\times\tilde{\mathbb{T}}_2$,
inequality (\ref{in3}) implies that
$$u^p(t_1,t_2)\leq a(t_1,t_2)\left(1+\int_{a_1}^{t_1}\int_{a_2}^{t_2}
f(s_1,s_2)\frac{u^q(s_1,s_2)}{a(s_1,s_2)}\Delta_1 s_1\Delta_2
s_2\right).$$  Define $v(t_1,t_2)$ on
$\tilde{\mathbb{T}}_1\times\tilde{\mathbb{T}}_2$ by

$$v(t_1,t_2)=1+\int_{a_1}^{t_1}\int_{a_2}^{t_2}
f(s_1,s_2)\frac{u^q(s_1,s_2)}{a(s_1,s_2)}\Delta_1 s_1\Delta_2
s_2.$$ Then,
\begin{equation*}
\frac{\partial}{\Delta_2 t_2}\left(\frac{\partial
v(t_1,t_2)}{\Delta_1
t_1}\right)=f(t_1,t_2)\frac{u^q(t_1,t_2)}{a(t_1,t_2)}\leq
f(t_1,t_2)a^{\frac{q}{p}-1}(t_1,t_2)v^{\frac{q}{p}}(t_1,t_2) \, ,
\end{equation*}
and noting that $v^{\frac{q}{p}}(t_1,t_2)\leq v(t_1,t_2)$ we conclude that
$$\frac{\partial}{\Delta_2 t_2}\left(\frac{\partial
v(t_1,t_2)}{\Delta_1 t_1}\right)\leq
f(t_1,t_2)a^{\frac{q}{p}-1}(t_1,t_2)v(t_1,t_2).$$ We
can now follow the same procedure as in the proof of Theorem~\ref{teo1} to obtain
$$v(t_1,t_2)\leq e_p(t_1,a_1),$$
where
$p(t_1)=\int_{a_2}^{t_2}f(t_1,s_2)a^{\frac{q}{p}-1}(t_1,s_2)\Delta_2
s_2$. Noting that $$u(t_1,t_2)\leq
a^{\frac{1}{p}}(t_1,t_2)v^{\frac{1}{p}}(t_1,t_2),$$ we obtain the
desired inequality (\ref{in4}).
$\>\Box$ \end{proof}

\begin{theorem}
Let $u(t_1,t_2), a(t_1,t_2), f(t_1,t_2)\in
C(\tilde{\mathbb{T}}_1\times\tilde{\mathbb{T}}_2,\mathbb{R}_0^+)$,
with $a$ and $f$ nondecreasing in each of the variables and
$g(t_1,t_2,s_1,s_2)\in C(S,\mathbb{R}_0^+)$, where
$S=\{(t_1,t_2,s_1,s_2)\in\tilde{\mathbb{T}}_1
\times\tilde{\mathbb{T}}_2\times\tilde{\mathbb{T}}_1
\times\tilde{\mathbb{T}}_2:a_1\leq
s_1\leq t_1,a_2\leq s_2\leq t_2\}$. If $p$ and $q$ are two
positive real numbers such that $p\geq q$ and if
\begin{equation}
\label{eq:6:sec}
u^p(t_1,t_2)\leq
a(t_1,t_2)+f(t_1,t_2)\int_{a_1}^{t_1}\int_{a_2}^{t_2}
g(t_1,t_2,s_1,s_2)u^q(s_1,s_2)\Delta_1 s_1\Delta_2 s_2
\end{equation}
for all $(t_1,t_2)\in
\tilde{\mathbb{T}}_1\times\tilde{\mathbb{T}}_2$, then
\begin{equation*}
u(t_1,t_2)\leq
a^{\frac{1}{p}}(t_1,t_2)\left[e_{\int_{a_2}^{t_2}f(t_1,t_2)
a^{\frac{q}{p}-1}(t_1,s_2)g(t_1,t_2,t_1,s_2)\Delta_2
s_2}(t_1,a_1)\right]^{\frac{1}{p}}
\end{equation*}
for all $(t_1,t_2)\in
\tilde{\mathbb{T}}_1\times\tilde{\mathbb{T}}_2$.
\end{theorem}

\begin{proof}
Since $a(t_1,t_2)$ is positive and nondecreasing on
$(t_1,t_2)\in\tilde{\mathbb{T}}_1\times\tilde{\mathbb{T}}_2$,
inequality (\ref{eq:6:sec}) implies that
$$u^p(t_1,t_2)\leq a(t_1,t_2)\left(1+f(t_1,t_2)\int_{a_1}^{t_1}\int_{a_2}^{t_2}
g(t_1,t_2,s_1,s_2)\frac{u^q(s_1,s_2)}{a(s_1,s_2)}\Delta_1
s_1\Delta_2 s_2\right).$$  Fix $t_1^\ast\in\tilde{\mathbb{T}}_1$
and $t_2^\ast\in\tilde{\mathbb{T}}_2$ arbitrarily and define a
function $v(t_1,t_2)$ on
$[a_1,t_1^\ast]\cap\tilde{\mathbb{T}}_1\times[a_2,t_2^\ast]\cap\tilde{\mathbb{T}}_2$
by
$$v(t_1,t_2)=1+f(t_1^\ast,t_2^\ast)\int_{a_1}^{t_1}\int_{a_2}^{t_2}
g(t_1^\ast,t_2^\ast,s_1,s_2)\frac{u^q(s_1,s_2)}{a(s_1,s_2)}\Delta_1
s_1\Delta_2 s_2.$$ Then,
\begin{align*}
\frac{\partial}{\Delta_2 t_2}\left(\frac{\partial
v(t_1,t_2)}{\Delta_1
t_1}\right)&=f(t_1^\ast,t_2^\ast)g(t_1^\ast,t_2^\ast,t_1,t_2)\frac{u^q(t_1,t_2)}{a(t_1,t_2)}\\
&\leq
f(t_1^\ast,t_2^\ast)g(t_1^\ast,t_2^\ast,t_1,t_2)a^{\frac{q}{p}-1}(t_1,t_2)v^{\frac{q}{p}}(t_1,t_2).
\end{align*}
Since $v^{\frac{q}{p}}(t_1,t_2)\leq v(t_1,t_2)$, we have that
$$\frac{\partial}{\Delta_2 t_2}\left(\frac{\partial
v(t_1,t_2)}{\Delta_1 t_1}\right)\leq
f(t_1^\ast,t_2^\ast)g(t_1^\ast,t_2^\ast,t_1,t_2)a^{\frac{q}{p}-1}(t_1,t_2)v(t_1,t_2).$$
We can follow the same steps as done before to reach the inequality
$$\frac{\partial v(t_1,t_2)}{\Delta_1 t_1}\leq\int_{a_2}^{t_2}f(t_1^\ast,t_2^\ast)g(t_1^\ast,
t_2^\ast,t_1,s_2)a^{\frac{q}{p}-1}(t_1,s_2)\Delta_2 s_2 v(t_1,t_2).$$
Fix $t_2=t_2^\ast$ and put
$p(t_1):=\int_{a_2}^{t_2^\ast}f(t_1^\ast,t_2^\ast)
g(t_1^\ast,t_2^\ast,t_1,s_2)a^{\frac{q}{p}-1}(t_1,s_2)\Delta_2
s_2\in \mathcal{R}^+$. Again, an application of \cite[Theorem~5.4]{inesurvey} gives
$$v(t_1,t_2^\ast)\leq e_p(t_1,a_1),$$
and putting $t_1=t_1^\ast$ we obtain the desired inequality.
$\>\Box$
\end{proof}

We end this section by considering a particular time scale. Let
$\{\alpha_k\}_{k\in\mathbb{N}}$ be a sequence of positive numbers
and let
$$
t_0^\alpha\in\mathbb{R} \, , \quad
t_k^\alpha=t_0^\alpha+\sum_{n=1}^k\alpha_n,\ k\in\mathbb{N} \, ,
$$
where we assume that $\lim_{k\rightarrow\infty}t_k^\alpha=\infty$. Then, we define the following time scale:
$\mathbb{T}^\alpha=\{t_k^\alpha:k\in\mathbb{N}_0\}$. For
$p\in\mathcal{R}$ we have (\textrm{cf.} \cite[Example~4.6]{Agarwal}):
\begin{equation}
\label{e1} e_p(t_k^\alpha,t_0^\alpha)=\prod_{n=1}^k(1+\alpha_n
p(t_{n-1})),\ \mbox{for all}\ k\in\mathbb{N}_0.
\end{equation}
Given two sequences $\{\alpha_k,\beta_k\}_{k\in\mathbb{N}}$ and
two numbers $t_0^\alpha,t_0^\beta\in\mathbb{R}$ as above, we
define the two time scales
$\mathbb{T}^\alpha=\{t_k^\alpha:k\in\mathbb{N}_0\}$ and
$\mathbb{T}^\beta=\{t_k^\beta:k\in\mathbb{N}_0\}$. We state now our last corollary:

\begin{corollary}
\label{cor:3.1}
Let $u(t,s)$, $a(t,s)$, and $f(t,s)$, defined on
$\mathbb{T}^\alpha\times\mathbb{T}^\beta$, be nonnegative with
$a$ and $f$ nondecreasing. Further, let $g(t,s,\tau,\xi)$, where
$(t,s,\tau,\xi)\in\mathbb{T}^\alpha\times
\mathbb{T}^\beta\times\mathbb{T}^\alpha\times\mathbb{T}^\beta$
with $\tau\leq t$ and $\xi\leq s,$ be a nonnegative function. If
$p$ and $q$ are two positive real numbers such that $p\geq q$ and
if
\begin{equation}
\label{eq:ab:not}
u^p(t,s)\leq
a(t,s)+f(t,s)\sum_{\tau\in[t_0^\alpha,t)}\sum_{\xi\in[t_0^\beta,s)}
\mu^\alpha(\tau)\mu^\beta(\xi)g(t,s,\tau,\xi)u^q(\tau,\xi)
\end{equation}
for all $(t,s)\in\mathbb{T}^\alpha\times\mathbb{T}^\beta$, where
$\mu^\alpha$ and $\mu^\beta$ are the graininess functions of
$\mathbb{T}^\alpha$ and $\mathbb{T}^\beta$, respectively, then
\begin{equation*}
u(t,s)\leq
a^{\frac{1}{p}}(t,s)\left[e_{\int_{t_0^\beta}^{s}f(t,s)
a^{\frac{q}{p}-1}(t,\xi)g(t,s,t,\xi)\Delta^\beta
\xi}(t,t_0^\alpha)\right]^{\frac{1}{p}}
\end{equation*}
for all $(t,s)\in\mathbb{T}^\alpha\times\mathbb{T}^\beta$, where
$e$ is given by (\ref{e1}).
\end{corollary}

\begin{remark}
In (\ref{eq:ab:not}) we are slightly
abusing on notation by considering
$[t_0^\alpha,t)=[t_0^\alpha,t)\cap\mathbb{T}^\alpha$ and
$[t_0^\beta,t)=[t_0^\beta,t)\cap\mathbb{T}^\beta$.
\end{remark}


\section{\normalsize An application}
\label{aplic}

Let us consider the partial delta dynamic equation
\begin{equation}
\label{diff1}
\frac{\partial}{\Delta_2 t_2}\left(\frac{\partial
u^2(t_1,t_2)}{\Delta_1
t_1}\right)=F(t_1,t_2,u(t_1,t_2))
\end{equation}
under given initial boundary conditions
\begin{equation}
\label{diff1:bc}
u^2(t_1,0)=g(t_1),\ u^2(0,t_2)=h(t_2),\ g(0)=0,\ h(0)=0,
\end{equation}
where we are assuming $a_1=a_2=0$,
$F\in
C(\tilde{\mathbb{T}}_1\times\tilde{\mathbb{T}}_2\times\mathbb{R}_0^+,\mathbb{R}_0^+)$,
$g\in C(\tilde{\mathbb{T}}_1,\mathbb{R}_0^+)$, $h\in
C(\tilde{\mathbb{T}}_2,\mathbb{R}^+_0)$, with $g$ and $h$
nondecreasing functions and positive on their domains except at zero.

\begin{theorem}
Assume that on its domain, $F$ satisfies $$F(t_1,t_2,u)\leq
t_2u.$$ If $u(t_1,t_2)$ is a solution of the IBVP (\ref{diff1})-(\ref{diff1:bc}) for $(t_1,t_2)\in\tilde{\mathbb{T}}_1\times\tilde{\mathbb{T}}_2$, then
\begin{equation}
\label{in7} u(t_1,t_2)\leq
\sqrt{(g(t_1)+h(t_2))}\left[e_{\int_{0}^{t_2}s_2
(g(t_1)+h(s_2))^{-\frac{1}{2}}\Delta_2
s_2}(t_1,0)\right]^{\frac{1}{2}}
\end{equation}
for $(t_1,t_2)\in\tilde{\mathbb{T}}_1\times\tilde{\mathbb{T}}_2$,
except at the point $(0,0)$.
\end{theorem}

\begin{proof}
Let $u(t_1,t_2)$ be a solution of the IBVP (\ref{diff1})-(\ref{diff1:bc}). Then, it satisfies the
following delta integral equation:
$$u^2(t_1,t_2)=g(t_1)+h(t_2)
+\int_0^{t_1}\int_0^{t_2}F(s_1,s_2,u(s_1,s_2))\Delta_1 s_1\Delta_2 s_2.$$
The hypothesis on $F$ imply that
$$u^2(t_1,t_2)\leq g(t_1)+h(t_2)+\int_0^{t_1}\int_0^{t_2}s_2 u(s_1,s_2)\Delta_1 s_1\Delta_2 s_2.$$
An application of Theorem~\ref{teo2} with
$a(t_1,t_2)=g(t_1)+h(t_2)$ and $f(t_1,t_2)=t_2$ gives (\ref{in7}).
$\>\Box$ \end{proof}


\renewcommand{\refname}


\end{document}